\documentclass[a4paper,12pt]{amsart}

\usepackage{epsfig}
\usepackage{amsthm,amsfonts}
\usepackage{amssymb,graphicx,color}
\usepackage[all]{xy}
\usepackage{cancel} 
\usepackage{verbatim}
\usepackage{hyperref}
\usepackage{enumitem}

\newtheorem{theorem}{Theorem}[section]
\newtheorem*{theorem*}{Theorem}
\newtheorem{lemma}[theorem]{Lemma}
\newtheorem{corollary}[theorem]{Corollary}
\newtheorem{proposition}[theorem]{Proposition}
\newtheorem{definition}[theorem]{Definition}
\newtheorem{example}[theorem]{Example}
\newtheorem{remark}[theorem]{Remark}

\newcommand{\R}{\mathbb{R}}

\newcommand{\C}{\mathbb{C}}

\newcommand{\Z}{\mathbb{Z}}

\newcommand{\N}{\mathbb{N}}

\newcommand{\D}{\mathbb{D}}

\newcommand{\Sp}{\mathbb{S}}

\newcommand{\x}{\textbf {x}}

\newcommand{\T}{\widetilde}

\renewcommand{\deg}[1]{{\rm degree}(#1)}

\newcommand{\Sing}[1]{{\rm Sing}(#1)}

\begin{document}

\title[Multiplicity, regularity and blow-spherical equivalence]
{Multiplicity, regularity and blow-spherical equivalence of complex analytic sets}

\author{J. Edson Sampaio}
\address{J. Edson Sampaio - 
Departamento de Matem\'atica, Universidade Federal do Cear\'a, Av.
Humberto Monte, s/n Campus do Pici - Bloco 914, 60455-760
Fortaleza-CE, Brazil} 
 \email{edsonsampaio@mat.ufc.br}



\begin{abstract}
This paper is devoted to study multiplicity and regularity as well as to present some classifications of complex analytic sets. We present an equivalence for complex analytical sets, namely blow-spherical equivalence and we receive several applications with this new approach. For example, we reduce to homogeneous complex algebraic sets a version of Zariski's multiplicity conjecture in the case of blow-spherical homeomorphism, we give some partial answers to the Zariski's multiplicity conjecture, we show that a blow-spherical regular complex analytic set is smooth and we give a complete classification of complex analytic curves.
\end{abstract}

\maketitle

\section{Introduction}

Recently, L. Birbrair, A. Fernandes and V. Grandjean in \cite{BirbrairFG:2017} (see also \cite{BirbrairFG:2012}) defined a new equivalence, namely blow-spherical equivalence, to study some properties of subanalytic sets such as, for example, to generalize thick-thin decomposition of normal complex surface singularity germs introduced in \cite{BirbrairNP:2014}. We with the aim to study multiplicity and regularity as well as to present some classifications of complex analytic sets, we define a weaker variation of the equivalence presented in \cite{BirbrairFG:2017}, namely also blow-spherical equivalence. Roughly speaking, two subset germs of Euclidean spaces are called blow-spherical equivalents, if their spherical modifications are homeomorphic and the homeomorphism induces homeomorphic tangent links. This equivalence, essentially, lives between topological equivalence and subanalytic bi-Lipschitz equivalence. We receive several applications with this new approach, such as a reduction for a version of Zariski's multiplicity conjecture in the case of blow-spherical homeomorphism, as blow-spherical regular complex analytic sets are smooth and we obtain a complete classification of complex analytic curves. 

The main motivation to study about multiplicity comes from the following problem proposed, in 1971, by O. Zariski (see \cite{Zariski:1971}):
\begin{enumerate}[leftmargin=*]\label{zariski}
\item[]{\bf Question A} Let $f,g:(\C^n,0)\to (\C,0)$ be two reduced complex analytic functions. If there is a homeomorphism $\varphi:(\C^n,V(f),0)\to (\C^n,V(g),0)$, then is it $m(V(f),0)=m(V(g),0)$?
\end{enumerate} 

More than 45 years later, the question above is still unsettled. However, there exist some partial results about it, for example, R. Ephraim in \cite{Ephraim:1976a} and D. Trotman in \cite{Trotman:1977} showed that the multiplicity is a $C^1$ invariant. O. Saeki in \cite{Saeki:1988} and Stephen Yau in \cite{Yau:1988} showed that the multiplicity is an invariant of the embedded topology in the case of isolated quasihomogeneous surfaces, Greuel in \cite{Greuel:1986} and O'Shea in \cite{OShea:1987} showed that the conjecture \ref{zariski} has a positive answer in the case of quasihomogeneous hypersurface families with isolated singularities. Several other authors showed partial results about versions of the Zariski's multiplicity conjecture, more recently, we can cite \cite{Abderrahmane:2016}, \cite{Eyral:2015}, \cite{EyralR:2015} \cite{EyralR:2016}, \cite{MendrisN:2005}, \cite{PlenatT:2013}, \cite{SaiaT:2004} and \cite{TrotmanS:2016}. In order to know more about this conjecture see the survey \cite{Eyral:2007}. 

More recently, the author joint with A. Fernandes, in the paper \cite{FernandesS:2016}, proved that the multiplicity of a complex analytic surface in $\C^3$ is a bi-Lipschitz invariant and W. Neumann and A. Pichon in \cite{NeumannP:2016} showed that the multiplicity is a bi-Lipschitz invariant in the case of normal complex analytic surfaces. The bi-Lipschitz invariance of the multiplicity is an advance about a problem that has been extensively studied in recent years, the complex analytic surface classification. The work of L. Birbrair, W. Neumann and A. Pichon in \cite{BirbrairNP:2014} is the most recent significant result on the Lipschitz geometry of singularities about classification of complex analytic surfaces, more specifically: they presented a classification for surfaces with the intrinsic metric module bi-Lipschitz homeomorphisms. More about classification of complex analytic surfaces can be found in \cite{BirbrairF:2008}, \cite{BirbrairFN:2008}, \cite{BirbrairFN:2009}, \cite{BirbrairFN:2010}, \cite{BirbrairFG:2012} and \cite{BirbrairFG:2017}. We have also recent studies about classification of complex analytic curves, see for example \cite{BirbrairFG:2015b}, \cite{HefezH:2011}, \cite{HefezH:2013} and \cite{HefezHHR:2015}.

Another subject of interest to many mathematicians is to know how simple is the topology of complex analytic sets. For example, if topological regularity implies analytic regularity. In general this does not occur, but Mumford in \cite{Mumford:1961} showed a result in this direction, it was stated as follows: \emph{a normal complex algebraic surface with a simply connected link at a point $x$ is smooth in $x$}. This was a pioneer work in topology of singular algebraic surfaces.  From a modern viewpoint this result can be seen as follows:   \emph{a topologically regular normal complex algebraic surface is smooth}. Since, in $\C^3$, a surface is normal if and only if it has isolated singularities, the result can be formulated as follows: \emph{topologically regular complex surface in $\C^3$,  with isolated singularity, is smooth}. The condition of singularity to be isolated is important, since $\{(x,y,z)\in\C^3; y^2=x^3\}$ is a topologically regular surface, but it is non-smooth.  There are also examples of non smooth surfaces in $\C^4$ with topologically regular isolated singularity, for example, E. V. Brieskorn in \cite{Brieskorn:1966} showed that $\{(z_0,z_1,z_2,z_3)\in\C^4; z_0^3=z_1^2+z_2^2+z_3^2\}$ is a topologically regular surface, but it is non-smooth, as well. However, N. A'Campo in \cite{Acampo:1973} and L\^e D. T. in \cite{Le:1973} showed that if $X$ is a complex analytic hypersurface in $\C^n$ such that $X$ is a topological submanifold, then $X$ is smooth. Recently, the author in \cite{Sampaio:2016} (see also \cite{BirbrairFLS:2016}), proved a version of the Mumford's Theorem, He showed that Lipschitz regularity in complex analytic sets implies smoothness.

In this paper, we deal with blow-spherical aspects related to the above Zariski's question. More precisely, we consider the questions below:
\begin{enumerate}[leftmargin=*]\label{conjecture}
\item[]{\bf Question  A1.} Let $X,Y\subset \C^n$ be two complex analytic sets. If $X$ and $Y$ are blow-spherical homeomorphic, then is it $m(X,0)=m(Y,0)$?
\end{enumerate}
\begin{enumerate}[leftmargin=*]\label{conjectureH}
\item[]{\bf Question  A2.} Let $X,Y\subset \C^n$ be two homogeneous complex algebraic sets. If $X$ and $Y$ are blow-spherical homeomorphic, then is it $m(X,0)=m(Y,0)$?
\end{enumerate}
In the Section \ref{section:blow-spherical} we define blow-spherical equivalence and we present some properties of this equivalence. In the Subsection \ref{section:examples}, we present some examples of blow-spherical equivalences and in the Section \ref{section:counter-examples} we show that, for example, the blow-spherical equivalence is different of the topological, intrinsic bi-Lipschitz and bi-Lipschitz equivalences. The others sections are devoted for applications of the results of the Section \ref{section:blow-spherical}.

In the Section \ref{section:regularity}, we prove a version of the Mumford's Theorem. Namely, we show that if a complex analytic set $X$ is blow-spherical regular, then $X$ is smooth. No restriction on the dimension or co-dimension is needed. No restriction of singularity to be isolated is needed. How an application, we obtain the main result of \cite{BirbrairFLS:2016}, about Lipschitz regularity of complex analytic sets.

The Section \ref{section:multiplicity} is dedicated for studies about the invariance of the multiplicity. In the Subsection \ref{section:reduction}, we prove the Theorem \ref{reduction} that says: The Question A1 has a positive answer if, and only if, the Question A2 has a positive answer. 
In the Subsection \ref{subsec:mult_dim_one}, the Theorem \ref{mult_dim_one}, shows that the Question A2 has positive answer for hypersurface singularities whose your irreducible components have singular sets with dimension $\leq$ 1. In particular, in the Corollary \ref{mult_surface}, we prove the blow-spherical invariance of the multiplicity of complex analytic surface (not necessarily isolated) singularities in $\C^3$. Moreover, in the Subsection \ref{subsec:mult_aligned}, we prove that the Question A2 has a positive answer in the case of aligned singularities and in the Subsection \ref{subsec:mult_families}, we prove that the Question A2 has a positive answer in the case of families of hypersurfaces.

Finally, in the Section \ref{section:curves}, we present a complete classification for complex analytic curves in $\C^n$.

\bigskip

\noindent{\bf Acknowledgements}. I wish to thank my friends Alexandre Fernandes by incentive and interest in this research and also by your support and help, Vincent Grandjean by useful comments about the manuscript and Lev Birbrair by incentive and interest in this research.

\section{Preliminaries}\label{section:preliminaries}
Let $f\colon (\C^n,0)\to (\C,0)$ be the germ of a reduced analytic function at origin with $f\not\equiv 0$. Let $(V(f),0)$ be the germ of the zero set of $f$ at origin. We recall \emph{the multiplicity of} $V(f)$ at origin, denoted by $m(V(f),0)$, is defined as following: we write
$$f=f_m+f_{m+1}+\cdots+f_k+\cdots$$ where each $f_k$ is a homogeneous polynomial of degree $k$ and $f_m\neq 0$. Then, $$m(V(f),0):= m,$$ 
see, for example, \cite{Chirka:1989} for a definition of multiplicity in high co-dimension.
\begin{definition}
Let $A\subset \R^n$ be a subanalytic set such that $x_0\in \overline{A}$.
We say that $v\in \R^n$ is a tangent vector of $A$ at $x_0\in\R^n$ if there is a sequence of points $\{x_i\}\subset A\setminus \{0\}$ tending to $x_0\in \R^n$ and there is a sequence of positive numbers $\{t_i\}\subset\R^+$ such that 
$$\lim\limits_{i\to \infty} \frac{1}{t_i}(x_i-x_0)= v.$$
Let $C(A,x_0)$ denote the set of all tangent vectors of $A$ at $x_0\in \R^n$. We call $C(A,x_0)$ the {\bf tangent cone} of $A$ at $x_0$.
\end{definition}
Notice that $C(A,x_0)$ is the cone $C_3(A,x_0)$ as defined by Whitney (see \cite{Whitney:1972}).

\begin{remark}
Follows from the curve selection lemma for subanalytic sets that, if $A\subset \R^n$ is a subanalytic set and $x_0\in \overline{A}$ is a non-isolated point, then 
\begin{eqnarray*}
C(A,x_0)=\{v;\, \exists\, \alpha:[0,\varepsilon )\to \R^n\,\, \mbox{s.t.}\,\, \alpha(0)=x_0,\, \alpha((0,\varepsilon ))\subset A\,\, \mbox{and}\,\,\\ 
\alpha(t)-x_0=tv+o(t)\}. 
\end{eqnarray*}
\end{remark}
\begin{remark}
If $A\subset \C^n$ is a complex analytic set such that $x_0\in A$ then $C(A,x_0)$ is the zero set of a set of homogeneous polynomials (See \cite{Whitney:1972}, Chapter 7, Theorem 4D). In particular, $C(A,x_0)$ is the union of complex line passing through at the origin.
\end{remark}
Another way to present the tangent cone of a subset $X\subset\R^{n}$ at the origin $0\in\R^{n}$ is via the spherical blow-up of $\R^{n}$ at the point $0$. Let us consider the {\bf spherical blowing-up} (at origin) of $\R^{n}$ 
$$
\begin{array}{ccl}
\rho_n\colon\mathbb{S}^{n-1}\times [0,+\infty )&\longrightarrow & \R^{n}\\
(x,r)&\longmapsto &rx
\end{array}
$$

Note that $\rho_n\colon\mathbb{S}^{n-1}\times (0,+\infty )\to \R^{n}\setminus \{0\}$ is a homeomorphism with inverse mapping $\rho_n^{-1}\colon\R^{n}\setminus \{0\}\to \mathbb{S}^{n-1}\times (0,+\infty )$ given by $\rho_n^{-1}(x)=(\frac{x}{\|x\|},\|x\|)$. The {\bf strict transform} of the subset $X$ under the spherical blowing-up $\rho=\rho_n$ is $X':=\overline{\rho_n^{-1}(X\setminus \{0\})}$. The subset $X'\cap (\mathbb{S}^{n-1}\times \{0\})$ is called the {\bf boundary} of $X'$ and it is denoted by $\partial X'$. 
\begin{remark}
{\rm If $X\subset \R^{n}$ is a subanalytic set, then $\partial X'=\mathbb{S}_0X\times \{0\}$, where $\mathbb{S}_0X=C(X,0)\cap \mathbb{S}^{n-1}$.}
\end{remark}

\begin{definition}
Let $X\subset \R^n$ be a subanalytic set such that $0\in X$. We say that $x\in\partial X'$ is {\bf simple point of $\partial X'$}, if there is an open $U\subset \R^{n+1}$ with $x\in U$ such that:
\begin{itemize}
\item [a)] the connected components of $(X'\cap U)\setminus \partial X'$, say $X_1,..., X_r$, are topological manifolds with $\dim X_i=\dim X$, for all $i=1,...,r$;
\item [b)] $(X_i\cup \partial X')\cap U$ is a topological manifold with boundary, for all $i=1,...,r$;. 
\end{itemize}
Let $Smp(\partial X')$ be the set of all simple points of $\partial X'$.
\end{definition}
\begin{definition}
Let $X\subset \R^n$ be a subanalytic set such that $0\in X$.
We define 
$$k_X:Smp(\partial X')\to \N,$$ 
with $k_X(x)$ being the number of components of $\rho^{-1}(X\setminus\{0\})\cap U$, for $U$ an open sufficiently small with $x\in U$.
\end{definition}
For to know about subanalytic sets, see, for example, \cite{BierstoneM:1988}. 
\begin{remark}
It is clear that the function $k_X$ is locally constant. In fact, $k_X$ is constant in each connected component $C_j$ of $Smp(\partial X')$. Then, we define $k_X(C_j):=k_X(x)$ with $x\in C_j \cap Smp(\partial X')$.
\end{remark}
\begin{remark}
In the case that $X$ is a complex analytic set, there is a complex analytic set $\sigma$ with $\dim \sigma <\dim X$, such that $X_j\setminus \sigma$ intersect only one connected component $C_i$ (see \cite{Chirka:1989}, pp. 132-133), for each irreducible component $X_j$ of tangent cone $C(X,0)$, then we define $k_X(X_j):=k_X(C_i)$.
\end{remark}
\begin{remark}
The number $k_X(C_j)$ is equal the $n_j$ defined by Kurdyka e Raby \cite{Kurdyka:1989}, pp. 762 and also is equal the $k_j$ defined by Chirka in \cite{Chirka:1989}, pp. 132-133, in the case that $X$ is a complex analytic set. 
\end{remark}

We also remember a very useful result proved by Y.-N. Gau and J. Lipman in the paper \cite{Gau-Lipman:1983}.
\begin{lemma}[\cite{Gau-Lipman:1983}, p. 172, Lemma]\label{irredutivel}
Let $\varphi:A\to B$ be a homeomorphism between two complex analytic sets. If $X$ is an irreducible component of $A$, then $\varphi(X)$ is an irreducible component of $B$.
\end{lemma}
\begin{remark}
{\rm All the sets considered in the paper are supposed to be equipped with the Euclidean metric.}
\end{remark}

\section{Blow-spherical equivalence}\label{section:blow-spherical}

\begin{definition}
Let $(X, 0)$ and $(Y, 0)$ be subanalytic subsets germs, respectively at the origin of $\R^n$ and $\R^p$. 
\begin{itemize}
\item A continuous mapping $\varphi:(X,0)\to (Y,0)$, with $0\not\in \varphi(X\setminus \{0\})$, is a {\bf blow-spherical morphism} (shortened as {\bf blow-morphism}), if the mapping
$$
\rho_p^{-1}\circ \varphi\circ\rho _n:X'\setminus \partial X'\to Y'\setminus \partial Y'
$$
extends as a continuous mapping $\varphi':X'\to Y'$.
\item A {\bf blow-spherical homeomorphism} (shortened as {\bf blow-iso\-mor\-phism}) is a blow-morphism $\varphi:(X,0)\to (Y,0)$ such that the extension $\varphi'$ is a homeomorphism. In this case, we say that the germs $(X,0)$ and $(Y,0)$ are {\bf blow-spherical equivalents} or {\bf blow-spherical homeomorphic} (or {\bf blow-isomorphic}).
\end{itemize}
\end{definition}
The authors in \cite{BirbrairFG:2017} (see also \cite{BirbrairFG:2012}) defined blow-spherical homeomorphism with the additional hypotheses that the blow-spherical homeomorphism needs also to be subanalytic.
\subsection{Blow-spherical invariance of the relative multiplicities}\label{subsection:multiplicities}

\begin{proposition}\label{defi_diferencial}
If $X$ and $Y$ are blow-spherical homeomorphics, then $C(X,0)$ and $C(Y,0)$ are also blow-spherical homeomorphics.
\end{proposition}
\begin{proof}
Let $\varphi:X\to Y$ be a blow-isomorphism. 
Then $\varphi'|_{\partial X'}:\partial X'\to \partial Y'$ is a homeomorphism.
We define $d_0\varphi: C(X,0)\to C(Y,0)$ by
$$
d_0\varphi(x)=\left\{\begin{array}{ll}
\|x\|\cdot \nu _0\varphi(\frac{x}{\|x\|}),& x\not=0\\
0,& x=0,
\end{array}\right.
$$
where $\varphi'(x,0)=(\nu _0\varphi(x),0)$.
We have that $d_0\varphi$ is a blow-spherical homeomorphism, since $\rho^{-1}\circ d_0\varphi\circ \rho(x,t)=(\nu _0\varphi(x),t)$.

\end{proof}
\begin{theorem}\label{multiplicities}
Let $\varphi:(X,0)\to (Y,0)$ be a blow-spherical homeomorphism. If $C(X,0)=\bigcup_{j=1}^r X_j$ and $C(Y,0)=\bigcup_{j=1}^r Y_j$,  with $Y_j=d_0\varphi(X_j)$, then $k_X(X_j)=k_Y(Y_j)$, $j=1,...,r$.
\end{theorem}
\begin{proof}
Fixed $j\in \{1,...,r\}$, let $p\in \Sp_0 X_{j}\times \{0\}\subset \partial X'$ be a generic point and let $U \subset X'$ be a small neighborhood of $p$. Since $\varphi':X'\to Y'$ is a homeomorphism, we have that $V=\varphi'(U)$ is a small neighborhood of $\varphi'(p)\in \Sp_0 Y_{j}\times \{0\}\subset \partial Y'$. Moreover, $\varphi'(U\setminus \partial X')=V\setminus \partial Y'$, since $\varphi'|_{\partial X'}:\partial X'\to \partial Y'$ is a homeomorphism, as well. Using once more that $\varphi'$ is a homeomorphism, we obtain the number of connected components of $U\setminus \partial X'$ is equal to $V\setminus \partial Y'$, showing that $k_X(X_j)=k_Y(Y_j)$.
\end{proof}
\begin{remark}[see \cite{Chirka:1989}, p. 133, proposition]\label{multip}
{\rm Let $X\subset \C^n$ be a complex analytic set and $X_1,...,X_r$ the irreducible components of $C(X,0)$. Then, $m(X,0)=\sum_{j=1}^rk_X(X_j)m(X_j,0)$.}
\end{remark}

\subsection{Examples of blow-spherical equivalences}\label{section:examples}

\begin{proposition}\label{diferblow}
Let $X, Y\subset \R^m$ be two subanalytic sets. If $\varphi: X\to Y$ is a homeomorphism such that $\varphi$ and $\varphi^{-1}$are differentiable at the origin, then $\varphi$ is a blow-spherical homeomorphism.
\end{proposition}
\begin{proof}
Observe that $\nu_0\varphi:C(X,0)\cap \Sp^{m-1}\to C(Y,0)\cap\Sp^{m-1}$ given by $$\nu_0\varphi(\x)=\frac{D\varphi_0(\x)}{\|D\varphi_0(\x)\|}$$ 
is a homeomorphism with inverse 
$$(\nu_0\varphi)^{-1}(\x)=\frac{D\varphi^{-1}_0(\x)}{\|D\varphi^{-1}_0(\x)\|}.$$

Using that $\varphi(t\x)=tD\varphi_0(\x)+o(t)$, we obtain
$$
\lim\limits _{t\to 0^+}\frac{\varphi(t\x)}{\|\varphi(t\x)\|}=\frac{D\varphi_0(\x)}{\|D\varphi_0(\x)\|}=\nu_0\varphi(\x)
$$
Then the mapping $\varphi': X'\to Y'$ given by
$$
\varphi'(\x,t)=\left\{\begin{array}{ll}
\left(\frac{\varphi(t\x)}{\|\varphi(t\x)\|},\|\varphi(t\x)\|\right),& t\not=0\\
(\nu_0\varphi(\x),0),& t=0,
\end{array}\right.
$$
is a homeomorphism. Therefore, $\varphi$ is a blow-spherical homeomorphism.
\end{proof}

We do a slight digression to remind the notion of inner distance on a connected Euclidean subset.

Let $Z\subset\R^m$ be a path connected subset. Given two points $q,\tilde{q}\in Z$, we define the \emph{inner distance} in $Z$ between $q$ and $\tilde{q}$ by the number $d_Z(q,\tilde{q})$ below:
$$d_Z(q,\tilde{q}):=\inf\{ \mbox{length}(\gamma) \ | \ \gamma \ \mbox{is an arc on} \ Z \ \mbox{connecting} \ q \ \mbox{to} \ \tilde{q}\}.$$
Thus, we say that two subsets $X\subset \R^m$ and $Y\subset k$ are {\bf intrinsic bi-Lipschitz equivalents} (resp. {\bf bi-Lipschitz equivalents}) if there are a homeomorphism $\phi:X\to Y$ and $C>0$ such that $\frac{1}{C}d_X(x,y)\leq d_Y(\phi(x),\phi(y))\leq C d_X(x,y)$ (resp. $\frac{1}{C}\|x-y\|\leq \|\phi(x)-\phi(y)\|\leq C \|x-y\|$) for all $x,y\in X$.

The next proposition show another example of blow-spherical homeomorphism (see Proposition 4.1 in \cite{BirbrairFG:2012}).
\begin{proposition}\label{lip-blow}
Let $X, Y\subset \R^m$ be two subanalytic sets. If $\varphi:(X,0)\to (Y,0)$ is a subanalytic bi-Lipschitz homeomorphism (with respect to the metric induced). Then, $\varphi$ is a blow-spherical homeomorphism.
\end{proposition}
\begin{proof}
Observe that $\nu_0\varphi:C(X,0)\cap \Sp^{m-1}\to C(Y,0)\cap\Sp^{m-1}$ given by $$\nu_0\varphi(\x)=\frac{D_0\varphi(\x)}{\|D_0\varphi(\x)\|}$$ 
is a homeomorphism with inverse 
$$(\nu_0\varphi)^{-1}(\x)=\frac{D_0\varphi^{-1}(\x)}{\|D_0\varphi^{-1}(\x)\|},$$ 
where $D_0\varphi:C(X,0)\to C(Y,0)$ (resp. $D_0\varphi^{-1}:C(Y,0)\to C(X,0)$) is the Lipschitz derivative of $\varphi$ (resp. $\varphi^{-1}$) at the origin, as defined by A. Bernig and A. Lytichak in \cite{BernigL:2007}. We know that $D_0\varphi$ is given by
$$
D_0\varphi(v)=\lim\limits_{t\to 0^+} \frac{\varphi(\alpha(t))}{t},
$$
where $\alpha:[0,\varepsilon )\to X$ is a subanalytic curve such that $\alpha(t)=tv+o(t)$. Thus, writing $\varphi=(\varphi_1,...,\varphi_m)$, we define, for each $i=1,...,m$, $\Phi_i:\R^m\to \R$ by
$$
\Phi_i(y)=\inf \{\varphi_i(x)+C \|x-y\|;\, x\in X\},
$$
where $C$ is the Lipschitz constant of $\varphi$. Then $\Phi=(\Phi_1,...,\Phi_m):\R^m\to \R^m$ is a subanalytic Lipschitz map and it is a extension of $\varphi$. Moreover, if $v\in C(X,0)$, we have $D_0\varphi(v)=\lim\limits_{t\to 0^+} \frac{\Phi(t v)}{t}=D_0\Phi(v)$, since $\Phi(\alpha(t))-\Phi(tv)=o(t)$ if $\alpha(t)=tv+o(t)$. In particular, $\nu_0\varphi(v)=\frac{D_0\Phi(v)}{\|D_0\Phi(v)\|}$. Moreover, for $x_n\in X\setminus \{0\}$ such that $x_n\to 0$ and $\frac{x_n}{\|x_n\|}\to \x \in C(X,0)\cap \mathbb{S}^{m-1}$, we have
$$
\nu_0\varphi(\x)=\lim\limits _{n\to \infty }\frac{\varphi(x_n)}{\|\varphi(x_n)\|}=\lim\limits _{n\to \infty }\frac{\Phi(x_n)}{\|\Phi(x_n)\|},
$$
since $\lim\limits _{n\to \infty }\frac{\Phi(x_n)}{\|\Phi(x_n)\|}=\frac{D_0\Phi(\x)}{\|D_0\Phi(\x)\|}$.
Then the mapping $\varphi': X'\to Y'$ given by
$$
\varphi'(\x,t)=\left\{\begin{array}{ll}
\left(\frac{\varphi(t\x)}{\|\varphi(t\x)\|},\|\varphi(t\x)\|\right),& t\not=0\\
(\nu_0\varphi(\x),0),& t=0,
\end{array}\right.
$$
is a continuous map. Using $\varphi^{-1}$ instead of $\varphi$ above, we obtain that $\varphi$ is a blow-spherical homeomorphism.
\end{proof}

\subsection{Blow-spherical equivalence and other equivalences}\label{section:counter-examples}
Now we give some examples that separate blow-spherical equivalence of others equivalences.
\begin{example}
$X=\{(z,x_1,x_2,x_3)\in\C^4;\, z^3=x_1^2+x_2^2+x_3^2\}$ and $Y=\{(z,x_1,x_2,x_3)\in\C^4;\, z=0\}$ are topological equivalents (see \cite{Brieskorn:1966}). However, by Theorem \ref{reg-blow}, they are not blow-spherical equivalents.
\end{example}

\begin{example}
$X=\{(x,y)\in\C^2;\, y^2=x^3\}$ and $Y=\{(x,y)\in\C^2;\, y^2=x^5\}$ are blow-spherical equivalents, but they are not bi-Lipschitz equivalents (see \cite{Fernandes:2003}, Example 2.1).
\end{example}

\begin{example}
$X=\{(x,y)\in\C^2;\, y^2=x^3\}$ and $Y=\{(x,y)\in\C^2;\, y=0\}$ are intrinsic bi-Lipschitz equivalents, however by Theorem \ref{reg-blow}, they are not blow-spherical equivalents.
\end{example}

\section{Regularity of complex analytic sets}\label{section:regularity}

\begin{definition} A subset $X\subset\R^n$ is called {\bf blow-spherical regular} at $0\in X$ if there is an open neighborhood $U$ of $0$ in $X$ which is blow-spherical homeomorphic to an Euclidean ball. 
\end{definition}
We remember the Prill's Theorem proved in \cite{Prill:1967} .
\begin{lemma}[\cite{Prill:1967}, Theorem]\label{prill}
Let $C\in \C^n$ be a complex cone which is a topological manifold. Then $C$ is a plane in $\C^n$.
\end{lemma}
\begin{theorem}\label{reg-blow}
Let $X\subset \C^n$ be a complex analytic set. If $X$ is blow-spherical regular at $0\in X$, then $(X,0)$ is smooth.
\end{theorem}
\begin{proof}
By Theorem \ref{defi_diferencial}, $C(X,0)$ is blow-spherical homeomorphic to $\C^k$, where $k=\dim X$, then $C(X,0)$ is irreducible, since $\C^k$ is irreducible. In particular, $C(X,0)$ is a topological manifold. By Prill's Theorem, $C(X,0)$ is a plane. Hence, $m(C(X,0),0)=1$ and by Theorem \ref{multiplicities}, $k_X(C(X,0))=1$ and using the Remark \ref{multip}, $m(X,0)=k_X(C(X,0))\cdot m(C(X,0),0)=1$. Then, $(X,0)$ is smooth.
\end{proof}


\begin{definition} A subset $X\subset\R^n$ is called {\bf Lipschitz regular} (respectively {\bf subanalytically Lipschitz regular}) at $x_0\in X$ if there is an open neighborhood $U$ of $x_0$ in $X$ which is bi-Lipschitz homeomorphic (respectively subanalytic bi-Lipschitz homeomorphic) to an Euclidean ball. 
\end{definition}
In the paper \cite{BirbrairFLS:2016}, the authors defined the notion of Lipschitz regular complex analytic sets as the sets being subanalytically Lipschitz regular as in the definition above.

Now, we give another proof of the main theorem of \cite{BirbrairFLS:2016}.
\begin{corollary}[\cite{BirbrairFLS:2016}]
If $X\subset \C^n$ is a complex analytic set and subanalytically Lipschitz regular at $0$, then $(X,0)$ is smooth.
\end{corollary}
\begin{proof}
By Proposition \ref{lip-blow}, $X$ is blow-spherical regular, then by Theorem \ref{reg-blow}, $(X,0)$ is smooth.
\end{proof}

\section{Invariance of the multiplicity}\label{section:multiplicity}
In this Section, we give some partial answers to versions of the Zariski's multiplicity conjecture.
\subsection{Reduction of the Zariski's Conjecture to homogeneous algebraic sets}\label{section:reduction}

\begin{theorem}[Reduction for homogeneous sets]\label{reduction}
The Question A1 has a positive answer if, and only if, the Question A2 has a positive answer.
\end{theorem}
\begin{proof} 
Obviously, we just need to prove that a positive answer to the Question A2 implies a positive answer to the Question A1. Let $X,Y\subset \C^n$ be two complex analytic set and $\varphi:(X,0)\to (Y,0)$ be a blow-spherical homeomorphism. Let us denote by $X_1,\dots,X_r$ and $Y_1,\dots,Y_s$ the irreducible components of the tangent cones $C(X,0)$ and $C(Y,0)$ respectively. It comes from Proposition \ref{defi_diferencial} and Theorem \ref{multiplicities} that $r=s$ and the blow-spherical homeomorphism $d_0\varphi:C(X,0)\rightarrow C(Y,0)$, up to re-ordering of indexes, sends $X_i$ onto $Y_i$ and $k_X(X_i)=k_Y(Y_i)$ $\forall$ $i$.

We know that $X_i$ and $Y_i$ are irreducible homogeneous algebraic sets. Since the Question A2 has a positive answer, we get $m(X_i,0)=m(Y_i,0)$ $\forall \ i$. Finally, using the Remark \ref{multip}, we obtain $m(X,0)=m(Y,0)$.
\end{proof}

\subsection{Multiplicity of analytic sets with 1-dimensional singular set}\label{subsec:mult_dim_one}

Let $f:\C^n \to \C$ be a homogeneous polynomial with $\deg{f}=d$. We recall the map $\phi:\Sp^{2n-1}\setminus f^{-1}(0)\to \Sp^1$ given by $\phi(z)=\frac{f(z)}{|f(z)|}$ is a locally trivial fibration (see \cite{Milnor:1968}, \S 4). Notice that, $\psi:\C^n\setminus f^{-1}(0)\to \C\setminus \{0\}$ defined by $\psi(z)=f(z)$ is a locally trivial fibration such that its fibers are diffeomorphic the fibers of $\phi$.  Moreover, we can choose as geometric monodromy the homeomorphism $h_f:F_f\to F_f$ given by $h_f(z)=e^{\frac{2\pi i}{d}}\cdot z$, where $F_f:=f^{-1}(1)$ is the (global) Milnor fiber of $f$ (see \cite{Milnor:1968}, \S 9).

In the proof of the Theorem 2.2 in \cite{FernandesS:2016} was proved the following result.
\begin{proposition}\label{acampo_type}
Let $f,g:\C^{n+1}\to \C$ be two reduced homogeneous complex polynomials. If $\varphi:(\C^n,V(f),0)\to (\C^n,V(g),0)$ is a homeomorphism and $\chi (F_f)\not=0$, then $m(V(f),0)=m(V(g),0)$.
\end{proposition}

\begin{definition}
Let $f:\C^{n+1}\to \C$ be a complex polynomial with $\dim \Sing{V(f)}=1$ and $\Sing{V(f)}=C_1\cup...\cup C_r$. Then $b_i(f)$ denotes the $i$-th Betti number of the Milnor fiber of $f$ at the origin, $\mu_j'(f)$ is the Milnor number of a generic hyperplane slice of $f$ at $x_j\in C_j\setminus \{0\}$ sufficiently close to the origin, write $\mu'(f)=\sum\limits_{i=1}^r\mu_i'(f)$.
\end{definition}
\begin{theorem}\label{mult_dim_one}
Let $f,g:(\C^{n+1},0)\to (\C,0)$ be two reduced homogeneous complex polynomials such that $\dim\Sing{V(f_i),0}\leq 1$, where $f=f_1\cdots f_r$ is the decomposition of $f$ in irreducible polynomials. If there exists a homeomorphism $\varphi:(\C^{n+1},V(f),0)\to (\C^{n+1},V(g),0)$, then $m(V(f),0)=m(V(g),0)$.
\end{theorem}
\begin{proof}
By Proposition \ref{acampo_type}, we can suppose $\chi (F_f)=\chi (F_g)=0$ and by additivity of the multiplicity, we can suppose that $f$ and $g$ are irreducible homogeneous polynomials with degree $d$ and $k$, respectively. In addition, by A'Campo-L\^e's Theorem, we can suppose that $d,k>1$. In particular, 
$$\dim \Sing{V(f)}=\dim \Sing{V(g)}.$$

If $\dim \Sing{V(f)}=1$, then by Theorem 5.11 in \cite{Randell:1979}, we have
$$(d-1)^{n}- \mu'(d-1)+(-1)^{n-1}- \mu'(f)=0$$
and
$$(k-1)^{n}- \mu'(k-1)+(-1)^{n-1}- \mu'(g)=0.$$
Thus, we define the polynomial $P:\R\to \R$ by 
$$
P(t)=t^{n}- \mu'(f)t+(-1)^{n-1} - \mu'(f), \quad \forall t\in \R.
$$
Since $\mu'(f)=\mu'(g)$ (see \cite{Le:1973b}, Proposition and Th\'er\`eme 2.3), then $d-1$ and $k-1$ are zeros of polynomial of $P(t)$.
By Descartes' Rule, the polynomial $P(t)$ has at most one positive zero, since $\mu'(f)=\mu'(g)\geq 1$. Thus, $d=k$.

If $\dim \Sing{V(f)}=0$, let $\T{f},\T{g}:\C^n\times \C\to \C$ given by $\T{f}(z,z_{n+1})=f(z)$ and $\T{g}(z,z_{n+1})=g(z)$. It is easy to see that $m(V(\T{f}),0)=m(V(f),0)$ and $m(V(\T{g}),0)=m(V(g),0)$. Moreover,
$V(\T{f})=V(f)\times \C$ and $V(\T{g})=V(g)\times \C$, then we define $\T{\varphi}:(\C^n\times \C,V(\T{f}),0)\to (\C^n\times \C,V(\T{g}),0)$ by $\T{\varphi}(z,z_{n+1})=(\varphi(z),z_{n+1})$. We have that $\T{\varphi}$ is a homeomorphism. Therefore, by first part of this prove, $m(V(\T{f}),0)=m(V(\T{g}),0)$ and this finish the proof.
\end{proof}

\begin{corollary}\label{mult_surface}
Let $f,g:(\C^3,0)\to (\C,0)$ be two reduced homogeneous complex polynomials. If $\varphi:(\C^3,V(f),0)\to (\C^3,V(g),0)$ is a homeomorphism, then $m(V(f),0)=m(V(g),0)$.
\end{corollary}
Let $f\colon (\C^n,0)\to (\C,0)$ be the germ of an analytic function at origin. If we write 
$$f=f_m+f_{m+1}+\cdots+f_k+\cdots$$ 
where each $f_k$ is a homogeneous polynomial of degree $k$ and $f_m\neq 0$, we define ${\bf in} (f):=f_m$. Thus, by Theorem \ref{reduction}, we have the following consequences. 

\begin{corollary}\label{mult_dim_one_blow}
Let $f,g:(\C^{n+1},0)\to (\C,0)$ be two reduced complex analytic functions such that $\dim\Sing{V(f_i),0}\leq 1$, where ${\bf in} (f)=f_1\cdots f_r$ is the decomposition of ${\bf in} (f)$ in irreducible polynomials. If $\varphi:(\C^{n+1},V(f),0)\to (\C^{n+1},V(g),0)$ is a blow-spherical homeomorphism, then $m(V(f),0)=m(V(g),0)$.
\end{corollary}

\begin{corollary}\label{mult_surface_blow}
Let $f,g:(\C^3,0)\to (\C,0)$ be two reduced complex analytic functions. If $\varphi:(\C^3,V(f),0)\to (\C^3,V(g),0)$ is a blow-spherical homeomorphism, then $m(V(f),0)=m(V(g),0)$.
\end{corollary}

\subsection{Multiplicity of aligned singularities}\label{subsec:mult_aligned}
\begin{definition}
If $h : U \to \C$ is an analytic function, a good stratification for $h$ at a point $p \in V(h)$ is an analytic stratification, $S=\{S_{\alpha}\}$, of the hypersurface $V(h)$ in a neighborhood, $U$, of $p$ such that the smooth part of $V(h)$ is a stratum and so that the stratification satisfies Thom's $a_h$ condition with respect to $U \setminus  V(h)$. That is, if $q_i$ is a sequence of points in $U \setminus  V(h)$ such that $q_i \to q\in S_{\alpha}$ and $T_{q_i} V(h - h(q_i))$ converges to some hyperplane $T$, then $T_qS_{\alpha} \subset T$. 
\end{definition}
\begin{definition}
If $h : (U,0) \to (\C,0)$ is an analytic function, then an aligned good stratification for $h$ at the origin is a good stratification for $h$ at the origin in which the closure of each stratum of the singular set is smooth at the origin. If such an aligned good stratification exists, we say that $h$ has an aligned singularity at the origin. 
\end{definition}
\begin{theorem}\label{mult_aligned}
Let $f,g:(\C^n,0)\to (\C,0)$ be two reduced homogeneous complex polynomials. Suppose that $f$ and $g$ have an aligned singularity at the origin.  
If $\varphi:(\C^n,V(f),0)\to (\C^n,V(g),0)$ is a homeomorphism, then $m(V(f),0)=m(V(g),0)$.
\end{theorem}
\begin{proof}
We can suppose that $d=m(V(f),0)>1$ and $k=m(V(g),0)>1$. By Corollary 4.7 in \cite{Massey:1995}, we have
$$
(d-1)^n=\sum\limits_{i=1}^s\lambda_{f,z}^i(0)(d-1)^i
$$
and 
$$
(k-1)^n=\sum\limits_{i=1}^s\lambda_{g,z}^i(0)(k-1)^i,
$$
where $s=\dim \Sing{V(f)}=\dim \Sing{V(g)}$ and $\lambda_{f,z}^0(0),\cdots,\lambda_{f,z}^s(0)$ (resp. $\lambda_{g,z}^0(0)$,$\cdots$, $\lambda_{g,z}^s(0)$) are the L\^e's numbers of $f$ (resp. $g$) at the origin (see the definition and some properties of the L\^e's numbers in \cite{Massey:1995}). By Corollary 7.8 in \cite{Massey:1995}, we obtain $\lambda_{f,z}^i(0)=\lambda_{g,z}^i(0)$, for $i=0,...,s$. Then $d-1$ and $k-1$ are zeros of the following equation 
\begin{equation}\label{eq_le}
t^n-\sum\limits_{i=1}^s\lambda_z^it^i=0,
\end{equation}
where $\lambda_z^i:=\lambda_{f,z}^i(0)$, for $i=0,...,s$. By Descartes' Rule, the equation (\ref*{eq_le}) has only one positive zero, since $\lambda_z^i\geq 0$, for $i=0,...,s$. Then $d-1=k-1$, i.e., $m(V(f),0)=m(V(g),0)$.
\end{proof}

\begin{corollary}\label{mult_aligned_blow}
Let $f,g:(\C^n,0)\to (\C,0)$ be two reduced complex analytic function. Suppose that $f_1,...,f_s,g_1,...,g_s$ have an aligned singularity at the origin, where ${\bf in} (f)=f_1\cdots f_s$ (resp. ${\bf in} (g)=g_1\cdots g_s$) is the decomposition of ${\bf in} (f)$ (resp. ${\bf in} (f)$) in irreducible polynomials.  
If $\varphi:(\C^n,V(f),0)\to (\C^n,V(g),0)$ is a blow-spherical homeomorphism, then $m(V(f),0)=m(V(g),0)$.
\end{corollary}

\subsection{The Question A2 in the case of families of hypersurfaces}\label{subsec:mult_families}
\begin{definition}
The family of complex analytic functions $\{f_t\}_{t\in [0,1]}$ (resp. the family of complex analytic hypersurfaces $\{V(f_t)\}_{t\in [0,1]}$) is said to be {\bf topologically $\mathcal{R}$-equi\-sin\-gular} (resp. {\bf topologically $V$-equisingular}) if there are an open $U\subset \C^n$ and a continuous map $\varphi:U\times [0,1] \to \C^n$ such that $\varphi_t:=\varphi(\cdot,t):U \to \varphi(U\times \{t\})$ is a homeomorphism, $\varphi(0,t)=0$ and $f_t=f_0\circ\varphi_t$ (resp. $\varphi_t(V(f_t))=V(f_0)$) for all $t\in [0,1]$.
\end{definition}
\begin{remark}
{\rm Changing $\varphi_t$ by $\varphi_t\circ \varphi_0^{-1}$, we can suppose that $\varphi_0=$id.}
\end{remark}

\begin{definition}
Let $\{f_t\}_{t\in [0,1]}$ be an analytic family of functions. We say that the family $\{f_t\}_{t\in [0,1]}$ (resp. $\{V(f_t)\}_{t\in [0,1]}$) is {\bf equimultiple} if ${\rm ord}_0 (f_0)={\rm ord}_0 (f_t)$ (resp. $m(V(f_t),0)=m(V(f_0),0)$) for all $t\in [0,1]$.
\end{definition}

\begin{lemma}[\cite{Ephraim:1976a}, Theorem 2.6]\label{teo1}
If the germ of $V=V(f)$ at origin is irreducible, then there exists $\varepsilon >0$ such that for any $0<r<\varepsilon$, $H_1(B_r\setminus V;\Z)\cong \Z$.
\end{lemma}
\begin{lemma}[\cite{Ephraim:1976a}, Theorem 2.7]\label{teo2}
Let $V$ be a hypersurface, and suppose that the germ of $V$ at the origin is irreducible. Let $f$ be an analytic function on $B_{\varepsilon}$ which generates the ideal of $V$ at all ponts of $B_{\varepsilon}$. Then $f_*:H_1(B_r\setminus V;\Z)\to H_1(\C\setminus\{0\};\Z)$ is a isomorphism for all $r$, $0<r<\varepsilon $.
\end{lemma}
In the next result, $U_1$ and $U_2$ are opens of $\C^n$, $V_1$ and $V_2$ are hypersurfaces of $\C^n$.

\begin{lemma}[\cite{Ephraim:1976a}, Theorem 2.8]\label{teo3}
Suppose $\varepsilon$ is chosen as above to serve for both $V_1$ and $V_2$. Assume
$$
\varphi : (U_1,V_1,0)\to (U_2,V_2,0)
$$
is a homeomorphism. Choose $0<r<\varepsilon$ and $0<s<\varepsilon$ such that $B_r\subset \varphi (B_{\varepsilon})$ and $\varphi(B_s)\subset B_r$. Then, $\varphi_*:H_1(B_s\setminus V_1;\Z)\to H_1(B_r\setminus V_2;\Z)$ is an isomorphism.
\end{lemma}

We do not know the answer to the Question A2, however we have the following results.

\begin{theorem}\label{teo_zariski_h}
Let $F:\C^n\times [0,1]\to \C$ be a (not necessarily continuous) subanalytic function. Suppose that for each $t\in [0,1]$, $f_t:=F(\cdot,t):\C^n\to \C$ is a (not necessarily reduced) complex homogeneous polynomial. If $\{V(f_t)\}_{t\in [0,1]}$ is a topologically $V$-equisingular family, then it is equimultiple.
\end{theorem}
\begin{proof}
Let $\varphi:U\times [0,1] \to \C^n$ be a continuous map such that $\varphi_t:=\varphi(\cdot,t):U \to \varphi(U\times \{t\})$ is a homeomorphism, $\varphi(0,t)=0$ and 
$$\varphi_t(V(f_t))=V(f_0),\quad \mbox{ for all }t\in [0,1].$$

Note that there is $v\in \C^n \setminus \bigcup \limits _{t\in [0,1]}V(f_t)$. In fact, we denote $V=F^{-1}(0)\subset \C^n\times [0,1]$ and $V_t=V(f_t)\times \{t\}$, then $V=\bigcup\limits _{t\in [0,1]} V_t$. Moreover, $V\setminus \Sing{V}$ is a smooth submanifold of $\C^n\times [0,1]=\R^{2n}\times [0,1]$, then for each $x\in V\setminus \Sing{V}$, there are a neighborhood $U_x$ and a diffeomorphim $\phi_x: B_2(0)\subset \R^{m}\to U_x$, where $m$ is the dimension of $V\setminus \Sing{V}$. Using the co-area formula, we obtain that
$$
\mathcal{H}^{2n+1}(V)=\int_V\|\nabla p(x)\|dx=\int_0^1\mathcal{H}^{2n}(V\cap p^{-1}(t))dt=0,
$$
where $p: \C^n\times [0,1]\to [0,1]$ is the canonical projection (here, $\mathcal{H}^{k}(X)$ denote the Hausdorff measure $k$-dimensional of the set $X$). The last equality is why $V(f_t)$ is a complex algebraic set with dimension $n-1$ and, in particular, $\mathcal{H}^{2n-1}(V(f_t))=\mathcal{H}^{2n}(V(f_t))=0$. Therefore, $m<2n+1$. 

Suppose that $m=2n$. Using the co-area formula once more, we obtain that
$$
\int_{\overline{B_1(0)}}\|\nabla h(x)\|dx=\int_{h( \overline{B_1(0)})}\mathcal{H}^{2n-1}(h^{-1}(t))dt=0,
$$
where $h=p\circ \phi_x:\overline{B_1(0)}\to [0,1]$. However, $\|\nabla h(x)\|\not\equiv 0$, then $\mathcal{H}^{2n}(\overline{B_1(0)})=0$, but this is a contradiction. Then, $m\leq 2n-1$ and, therefore, $\mathcal{H}^{2n}(V)=0$.

Thus, $\mathcal{H}^{2n}(\bigcup \limits _{t\in [0,1]}V(f_t))=0$, since the canonical projection $\pi:\C^n\times [0,1] \to \C^n$ is a Lipschitz map and $\pi(V)=\bigcup \limits _{t\in [0,1]}V(f_t)$. In particular, $\bigcup \limits _{t\in [0,1]}V(f_t)\subsetneq \C^n$. 

Let $L\subset \C^n$ be a complex line given by $L=\{\lambda v;\,\lambda\in \C\}$. Then $L\cap (\bigcup \limits _{t\in [0,1]}V(f_t))=\{0\}$. Moreover, by Lemma \ref{irredutivel}, we can suppose that $f_t$ is an irreducible polynomial, for all $t\in [0,1]$.

Fixed $t_0\in [0,1]$, choose $0<r,s< \varepsilon $ as in the Lemma \ref{teo3} and let $\gamma$ be a generator of $H_1(\mathbb{D}_L;\Z)$, where $\mathbb{D}_L:=\{z\in L;\, 0<\|z\|\leq \delta\}\subset B_r\cap B_s$ and $B_{\varepsilon }\subset U$. Then, $(f_{t_0}|_{\mathbb{D}_L})_*(\gamma)= \pm m(V(f_{t_0}),0)$ and $(f_0|_{\mathbb{D}_L})_*(\gamma)= \pm m(V(f_0),0)$. In particular, $i_*(\gamma)=\pm m(V(f_{t_0}),0)$, where $i:\mathbb{D}_L\to B_s\setminus V(f_{t_0})$ is the inclusion map, since $(f_{t_0})_*:H_1(B_s\setminus V(f_{t_0});\Z)\to H_1(\C\setminus \{0\};\Z)$ is an isomorphism. 

However, $(\varphi_{t_0})_*:H_1(B_s\setminus V(f_{t_0});\Z)\to H_1(B_{\varepsilon}\setminus V(f_0);\Z)$ is also an isomorphism, then $({\varphi_{t_0}})_*(i_*(\gamma))=\pm m(V(f_{t_0}),0)$. Therefore, 
\begin{equation}\label{mult_gamma}
(f_0\circ \varphi_{t_0}|_{\mathbb{D}_L})_*(\gamma)=\pm m(V(f_{t_0}),0),
\end{equation}
since $(f_0)_*:H_1(B_{\varepsilon}\setminus V(f_0);\Z)\to H_1(\C\setminus \{0\});\Z)$ is an isomorphism, as well.

\noindent {\bf Claim.} $f_0\circ \varphi_{t_0}|_{\mathbb{D}_L}$ is homotopic to $f_0|_{\mathbb{D}_L}$.

In fact, $L\cap V(f_t)=\{0\}$ for all $t\in [0,1]$ and, in particular, for each $t\in [0,1]$, $f_t(w)\not =0$ for all $w\in \mathbb{D}_L$. Thus, the function $H:\mathbb{D}_L\times [0,1]\to \C\setminus \{0\}$ given by $H(z,\lambda)=f_0\circ \varphi(z,\lambda t_0)$ is a homotopy between $f_0|_{\mathbb{D}_L}$ and $f_0\circ \varphi|_{\mathbb{D}_L}$, since $\varphi_{0}=id$. In particular, $(f_0\circ \varphi|_{\mathbb{D}_L})_*=(f_0|_{\mathbb{D}_L})_*$. Then
$$
\pm m(V(f_{t_0}),0)\overset{(\ref*{mult_gamma})}{=} (f_0\circ\varphi_{t_0}|_{\mathbb{D}_L})_*(\gamma)=(f_0|_{\mathbb{D}_L})_*(\gamma)= \pm m(V(f_0),0).
$$
Therefore, $m(V(f_{t_0}))= m(V(f_0))$.
\end{proof}

\begin{theorem}\label{teo_zariski_hf}
Let $F:\C^n\times [0,1]\to \C$ be a (not necessarily continuous) subanalytic function. Suppose that for each $t\in [0,1]$, $f_t:=F(\cdot,t):\C^n\to \C$ is a (not necessarily reduced) complex homogeneous polynomial. If $\{f_t\}_{t\in [0,1]}$ is a topologically $\mathcal{R}$-equisingular family, then it is equimultiple.
\end{theorem}
\begin{proof}
Let $\varphi:U\times [0,1] \to \C^n$ be a continuous map such that $\varphi_t:=\varphi(\cdot,t):U \to \varphi(U\times \{t\})$ is a homeomorphism, $\varphi(0,t)=0$ and $f_t=f_0\circ\varphi_t$ for all $t\in [0,1]$.

We have that there is $v\in \C^n \setminus \bigcup \limits _{t\in [0,1]}V(f_t)$.
Let $L\subset \C^n$ be a complex line given by $L=\{\lambda v;\,\lambda\in \C\}$. Then $L\cap (\bigcup \limits _{t\in [0,1]}V(f_t))=\{0\}$. 

Fixed $t_0\in [0,1]$, we have $\partial_{{\rm top}}(f_{t_0}|_{\mathbb{D}_L})= {\rm ord}_0 (f)$ and $\partial_{{\rm top}}(f_0|_{\mathbb{D}_L})={\rm ord}_0 (f_0)$, where $\mathbb{D}_L:=\{z\in L;\, 0<\|z\|\leq \delta\}$. Moreover, $\partial_{{\rm top}}(f_0\circ\varphi_{t_0}|_{\mathbb{D}_L})=\partial_{{\rm top}}(f_{t_0}|_{\mathbb{D}_L})= {\rm ord}_0 (f_{t_0})$, since $f_{t_0}=f_0\circ \varphi_{t_0}$.

\noindent {\bf Claim.} $f_{t_0}|_{\mathbb{D}_L}$ is homotopic to $f_0|_{\mathbb{D}_L}$.

In fact, $L\cap V(f_t)=\{0\}$ for all $t\in [0,1]$ and, in particular, for each $t\in [0,1]$, $f_t(w)\not =0$ for all $w\in \mathbb{D}_L$. Thus, the function $H:\mathbb{D}_L\times [0,1]\to \C\setminus \{0\}$ given by $H(z,\lambda)=f_0\circ \varphi(z,\lambda t_0)$ is a homotopy between $f_0|_{\mathbb{D}_L}$ and $f_{t_0}|_{\mathbb{D}_L}$, since $f_{t_0}=f_0\circ \varphi_{t_0}$ and $f_0=f_0\circ \varphi_0$. In particular, $\partial_{{\rm top}}(f|_{\mathbb{D}_L})=\partial_{{\rm top}}(f_0|_{\mathbb{D}_L})$.

Then, ${\rm ord}_0 (f_{t_0})= {\rm ord}_0 (f_0)$.
\end{proof}

\begin{remark}
We finish this Section remarking that the Corollaries \ref{mult_dim_one_blow}, \ref{mult_surface_blow}, and \ref{mult_aligned_blow} are true as well when we consider a bi-Lipschitz homeomorphism instead of a blow-spherical homeomorphism, using the Theorem 2.1 in \cite{FernandesS:2016} instead of the Theorem \ref{reduction}.
\end{remark}

\section{Classification of complex analytic curves in the space}\label{section:curves}

In this Section, we prove that the blow-spherical geometry and the analytic multiplicity are essentially the same object, in the case of complex analytic curves.
\begin{theorem}
Two irreducible analytic curves are blow-isomorphic if, and only if, they have the same multiplicity.
\end{theorem}
\begin{proof}
Let $X,\T{X}\subset\C^n$ be two irreducible analytic curves. By Theorem \ref{multiplicities}, we have that if $X$ and $\T{X}$ are blow-isomorphic, then they have the same multiplicity, since $m(C(X,0))=m(C(\T{X},0))=1$. Suppose that $k=m(X,0)=m(\T{X},0)$. After a change of coordinates linear, if necessary, we can suppose that the tangent cone of $X$ and of $\T{X}$ is $\{(\xi ,0)\in\C^n;\, \xi\in\C\}$. Let $\psi:\D_{\varepsilon }\to X$ and $\T{\psi }:\D_{\varepsilon }\to \T{X}$ the Puiseux's parametrizations of $X$ and $\T{X}$, resp., given by
$$
\psi(t)=(t^k,\phi(t))=(t^k,\phi_2(t),...,\phi_n(t))
$$
and
$$
\T{\psi}(t)=(t^k,\T{\phi}(t))=(t^k,\T{\phi}_2(t),...,\T{\phi}_n(t)),
$$
where $ord_0\phi_i>k$ and $ord_0\T{\phi}_i>k$, for $i=2,...,n$.
Define $\varphi:X\to \T{X}$ by $\varphi=\T{\psi}\circ \psi^{-1}$.
Let $z_m=(\x_m,t_m)\in X'\setminus \partial X'$ s.t. $z_m\to (\x,0)$. Then $\lim\limits _{m\to\infty }\varphi'(z_m)=(\x,0)$. In fact, with $s_m=\psi^{-1}(t_m\x_m)$, we have
\begin{eqnarray*}
 \varphi'(z_m) &=& \left(\frac{\varphi(t_m\x_m)}{\|\varphi(t_m\x_m)\|},\|\varphi(t_m\x_m)\|\right)\\
			   &=& \left(\frac{(s_m^k,\T{\phi}(s_m))}{\|(s_m^k,\T{\phi}(s_m))\|},\|(s_m^k,\T{\phi}(s_m))\|\right).
\end{eqnarray*}
But $t_m\x_m=(s_m^k,\phi(s_m))$ and, hence, 
$$
z_m=\left(\frac{(s_m^k,\phi(s_m))}{\|(s_m^k,\phi(s_m))\|},\|(s_m^k,\phi(s_m))\|\right).
$$
Then, $\lim\limits _{m\to\infty }\varphi'(z_m)=\lim\limits _{m\to\infty }z_m=(\x,0)$. Thus, $\varphi'$ extends continuously to identity on $\partial X'=\partial Y'$.
\end{proof}
Let $X,Y\subset \C^n$ be two complex analytic curves. Let $X_1,...,X_r\subset \C^n$ be the irreducible components of $X$ and let $Y_1,...,Y_s\subset \C^n$ be the irreducible components of $Y$. Then, we have the following
\begin{corollary}
$X$ and $Y$ are blow-spherical equivalents if and only if there is a bijection $\sigma:\{1,...,r\}\to\{1,...,s\}$ such that 
\begin{itemize}
\item [1)] $m(X_i,0)=m(Y_{\sigma(i)},0)$, for all $i=1,...,r$.
\item [2)] there is a homeomorphism $h:(C(X,0),0)\to (C(Y,0),0)$ satisfying $h(C(X_i,0))=C(Y_{\sigma(i)},0)$, for all $i=1,...,r$.
\end{itemize}
\end{corollary}

\end{document}